\documentclass[12pt]{article}
\usepackage{amssymb,amsmath}

\usepackage{amsfonts}
\usepackage{dsfont}
\usepackage{amsmath,amssymb,amsthm}
\usepackage{eucal}
\usepackage{latexsym}
\usepackage{amsbsy}
\usepackage{amsmath}
\usepackage{amscd}
\usepackage{amsgen}
\usepackage{amsopn}
\usepackage{amstext}
\usepackage{ifthen}
\usepackage{amsxtra}
\usepackage{amsfonts}
\usepackage{float}

\usepackage{hyperref}
\usepackage{ulem}
\usepackage{xcolor}

\def \R{{\hbox{\vrule width 0.6pt height 6.8pt depth -.2pt\kern-0.2pt
R}}}
\def \P{{\hbox{\vrule width 0.6pt height 6.8pt depth -.2pt\kern-0.2pt
P}}}

\usepackage[english]{babel}
\usepackage[latin1]{inputenc}
\usepackage{latexsym}

\def \R {\mathbb R}
\def \N {\mathbb N}
\def \P {\mathbb P}

\def\MOD#1{{|\kern -.16em |\kern -.16em | #1 | \kern -.16em |\kern
 -.16em |}}
\def \epsilon {\varepsilon}

\newtheorem{theo}{\bf THEOREM}[section]
\newtheorem{lem}[theo]{\bf LEMMA}
\newtheorem{pro}[theo]{\bf PROPOSITION}

\newtheorem{defi}[theo]{\bf DEFINITION}
\newtheorem{rem}[theo]{\bf REMARK}
\newcommand{\eps}{\varepsilon}
\newenvironment{proof_prop3.5}[1][\bf{Proof of Proposition  \ref{prop_red} }]{\noindent{\it{#1}}}{\hfill$\square$\\}

\newenvironment{proof_lemm3.8}[1][\bf{Proof of Lemma \ref{lem_est} }]{\noindent{\it{#1}}}{\hfill$\square$\\ }

\numberwithin{equation}{section}

\title{ \textbf{\Large Refined blow-up asymptotics
    for a perturbed nonlinear heat equation with
    a gradient and a non-local  term} }

\author{\large Bouthaina Abdelhedi\\
   \textit{\small Universit\'e  de Sfax, Facult\'e  des Sciences de Sfax, }\\
 \textit{\small D\'epartement de  Mathematiques,  BP 1171, Sfax 3000, Tunisie.}\\
   \large Hatem Zaag\\
  \textit {\small Universit\'e Sorbonne Paris Nord},\\
   \textit{\small LAGA, CNRS (UMR 7539), F-93430, Villetaneuse, France}
 }

\begin{document}
\maketitle

\medskip

\begin{abstract} We consider in this paper a perturbation of the
  standard semilinear heat equation by a term involving the space
  derivative and a non-local term. In some earlier works \cite{AZ1,
    AZ2}, we constructed a  solution $u$  for that equation  such that
  $u$ and $\nabla u$ both blow up at the origin and only there. We
  also gave the final blow-up  profile. In this paper,
we refine our construction method in order to get a sharper estimate
on the gradient at blow-up.
\end{abstract}
~~\\
\textbf{ AMS 2010 Classification:}  35B20, 35B44, 35K55.\\
\textbf{Keywords:} Blow-up, nonlinear heat equation, gradient term, non-local term.\\
\section{Introduction}
We consider in this paper   the following nonlinear parabolic equation\begin{equation}\label{eq_u}
\left \{
\begin{array}{lcl}
 u_{t}&=&\Delta u+|u|^{p-1}u+\mu |\nabla u|\displaystyle \int_{B(0, |x|)}|u|^{q-1},\\
 u(0)&=&u_0\in W^{1, \infty}(\R^N),
\end{array}
\right.
\end{equation}

where $u=u(x,t)\in \R$, $x\in \R^N$ and the parameters $p,\; q$ and $\mu$ are such that \begin{equation}\label{hyp}
\displaystyle  p>3, \quad \frac{N}{2}(p-1)+1<q<\frac{N}{2}(p-1) +\frac{p+1}{2},\quad \mu\in \R.
\end{equation}

Equation $(\ref{eq_u})$ is wellposed in the weighted  functional  space $W^{1, \infty}_\beta(\R^N)$ defined as follows:
\begin{equation}
 W^{1, \infty}_\beta(\R^N)=\{g, \; (1+|y|^\beta) g\in L^\infty, \;(1+|y|^\beta)\nabla  g\in L^\infty\},
 \end{equation}
where
\begin{equation}\label{beta}
 0\le \beta <\frac{2}{p-1},\;\mbox{if}\;\mu=0\; \mbox{ and }\;\displaystyle \frac{N}{q-1}<\beta<\frac{2}{p-1},\;\mbox{if}\; \mu\neq 0,
\end{equation}
as one may see from Appendix C in \cite{AZ1}. From the standard
dichotomy, either the maximal solution is global in time, or it exists
up to some maximal time $T<+\infty$ with
\[
\|u(t)\|_{W^{1,\infty}_\beta}\to \infty \mbox{ as }t\to + \infty.
\]
In that case, we say that $u(x,t)$ blows up in finite time, and we
call $T$ the blow-up time of the solution.

\medskip

When $\mu=0$, equation \eqref{eq_u} 
becomes
the standard semilinear heat equation with power nonlinearity:
\begin{equation}\label{equ}
u_t = \Delta u +|u|^{p-1}u.
\end{equation}
The existence of blow-up solutions for equation \eqref{equ} has been extensively studied,
see Fujita \cite{fujita}, Ball \cite{ball}, Berger and Kohn
\cite{berger}, Herrero and Vel\'azquez \cite{herrero3}, Bricmont and Kupiainen
\cite{bricmont}, Merle and Zaag \cite{MZ97} and the references therein.\\
 In particular, the authors in \cite{bricmont} and \cite{MZ97} constructed a solution $u$ which approaches an explicit universal profile $f$ depending only on $p$ and independent from initial data  as follows:
\begin{equation}
\left\|(T-t)^{\frac{1}{p-1}}
  u(x,t)-f\left(\frac{x}{\sqrt{(T-t)|\log(T-t)|}}\right)\right\|_{L^\infty}\to
0,
\end{equation}
as $t\to T$, where $f$ is the profile defined by
\begin{equation}\label{profil}
\displaystyle f(z)=\left(p-1+\frac{(p-1)^2}{4p}|z|^2\right)^{-\frac{1}{p-1}}.
\end{equation}
The proof relies on 2 parts:\\
- A formal part, where one finds an approximate solution, which will
be considered as a profile for the exact solution to be constructed.\\
- A rigorous part, where one linearizes the equation around the
approximate solution (i.e. the profile), and shows that the linearized solution has a
solution which converges to $0$. This part relies itself on 2 steps:
the first, where we reduce the control of the solution (which is
infinite dimensional) to a finite dimensional problem. Then, the
finite dimensional problem is solved thanks to index theory.

\medskip

This method has proved to be efficient for different PDEs from
different types, and no list can be exhaustive (see del Pino, Musso
and Wei \cite{DMWapde20}, Nouaili and Zaag \cite{NZarma18}, Tayachi
and Zaag \cite{TZ}, Duong and Zaag \cite{DZ}, 
  Mahmoudi, Nouaili and Zaag \cite{MNZ},  Merle, Raphael,  Rodnianski and  Szeftel \cite{MR},  Collot, Ghoul, Nguyen and Masmoudi \cite{CGNM}, etc...).

\medskip
In \cite{AZ1} and \cite{AZ2}, we considered equation \eqref{eq_u} as a
challenge for the construction of blow-up solutions, since it
presents a double difficulty: the gradient term and the nonlocal
term, and we were successful in proving the following:
%
%

\medskip

- In \cite{AZ1}, we constructed a solution $u(x,t)$ for  equation $(\ref{eq_u})$  which  blows up  in finite time $T$ at $a=0$, and we proved that the solution approaches the profile $f$ \eqref{profil} in the sense that
for all $(x, t)\in \R^N\times [0, T)$:
   \begin{eqnarray}\label{b_u}
  \displaystyle \Big|u(x,t)\!-\!(T\!-\!t)^{\!-\!\frac{1}{p-1}}f(\frac{x}{\sqrt{(T\!-\!t)|\log(T\!-\!t)|}})\Big|\!\leq \!\frac{C}{1\!+\!(\frac{|x|^2}{T\!-t})^{\frac{\beta}{2}}}\frac{(T\!-\!t)^{\!-\!\frac{1}{p-1}}}{|\log(T\!-\!t)|^{\frac{1\!-\!\beta}{2}}}, \end{eqnarray}
   and
\begin{eqnarray}\label{nabla_u}
\displaystyle \! \Big|\!\nabla \! u(x,t)\!-\!\frac{(T\!-\!t)^{\!-\!\frac{1}{2}\!-\!\frac{1}{p\!-\!1}}}{\sqrt{|\log(T\!-\!t)|}}\nabla \! f\!(\!\frac{x}{\sqrt{(T\!\!-\!t)|\log(T\!\!-\!t)|}})\Big|\!\!\leq\! \!\!\frac{C}{1\!\!+\!\!(\frac{|x|^2}{T\!-t})^{\!\frac{\beta}{2}}}\frac{(T\!-\!t)^{\!-\!\frac{1}{2}\!-\!\frac{1}{p-1}}}{|\log(\!T-\!t)|^{\frac{1\!-\!\beta}{2}}},
\end{eqnarray}
 (note that $0\le \beta<1$ from \eqref{beta} and \eqref{hyp}), and that $f$ is called  the ``intermediate'' profile).
 
 \medskip

- In \cite{AZ2}, adapting the technique developed for equation \eqref{equ} by Giga and Kohn \cite{giga1}, we prove  that neither $u$ nor $\nabla u$ blow up outside
  the origin.
Since $u$ blows up at the origin by \eqref{b_u}, this
  implies the single
point blow-up property for $u$. Using the mean value theorem, this
yields that $\nabla u$ blows up at the origin, which
 proves the single point blow-up property
for $\nabla u$ too. 
We also
prove the existence of a blow-up final profile $u^\ast$ such that $u(x,t)\to u^\ast(x)$ as $t\to T$ in $C^1$
of every compact of $\R^N\backslash \{0\}$. Next, we find an
equivalent of $u^\ast$ and an upper bound on $\nabla u^*$ near the
blow-up point:
\begin{equation}\label{profi_u}
u^\ast (x)\sim \Big [\frac{8p|\log|x||}{(p-1)^2|x|^2}\Big]^{\frac{1}{p-1}} ,\;\mbox{as}\; x\to 0,
\end{equation}
and for $|x|$ small,
\begin{equation}\label{profi_nabla}
\displaystyle |\nabla u^\ast(x)|\leq C |x|^{-\frac{p+1}{p-1}} \big|\log|x|\big|^{ \frac{p+3}{4(p-1)}}.
\end{equation}

\medskip

Our goal in this paper is to give a refined asymptotic description of
the blow-up solution. In fact, we prove more refined versions of \eqref{b_u}, \eqref{nabla_u} and \eqref{profi_nabla}.\\
This will be
  done through the introduction of a sharper shrinking set. More precisely,  we prove the following theorem:
\begin{theo}[A sharper description of blow-up]\label{th1}
Let $\mu\in \R$, $p>3,$  and $q\in \R$ such that $\displaystyle
\frac{N}{2}(p-1)+1<q<\frac{N}{2}(p-1) +\frac{p+1}{2}$.\\
Consider an arbitrary   $\beta$  such that
\begin{equation}
 0\leq \beta <\frac{2}{p-1},\;\mbox{if}\;\mu=0\; \mbox{ and }\;\displaystyle \frac{N}{q-1}<\beta<\frac{2}{p-1},\;\mbox{if}\; \mu\neq 0.
\end{equation}
Consider also some arbitrary  $\epsilon_1\in (0, \frac12]$ and $\alpha \in (0,  \frac12)$.\\
Then, there exists $T>0$ such that equation $(\ref{eq_u})$ has a solution $u(x,t)$ such that $u$ and $\nabla u $ simultaneously  blow up at time  $T$ at the point $a=0$. Moreover,
\begin{enumerate}
\item For all $t\in [0, T)$, for all $x\in \R^N$,
\begin{equation}\label{N_u}
\displaystyle |u(x,t)-(T-t)^{-\frac{1}{p-1}}f(\frac{x}{\sqrt{(T-t)|\log(T-t)|}})|\leq \frac{C}{1+(\frac{|x|^2}{T-t})^{\frac{\beta}{2}}}\frac{(T-t)^{-\frac{1}{p-1}}}{|\log(T-t)|^{1-\frac{\beta}{2}-\eps_1}},
\end{equation}
and
\begin{equation}\label{N_nabla}
\displaystyle\! \!|\nabla u(x,t)-\frac{(T\!-\!t)^{\!-\frac{1}{2}\!-\!\frac{1}{p\!-\!1}}}{\sqrt{|\log(T-t)|}}\nabla f(\frac{x}{\sqrt{(T\!-\!t)|\log(T\!-\!t)|}})|\!\leq \!\!\frac{C}{1\!+\!(\frac{|x|^2}{T-t})^{\frac{\beta}{2}}}\frac{(T\!-\!t)^{-\frac{1}{2}-\frac{1}{p-1}}}{|\log(T\!-\!t)|^{1\!-\!\frac{\beta}{2}\!-\!\eps_1}},
\end{equation}
where $f$ is defined in \eqref{profil}.

\item Blow-up occurs only at the origin.
\item  For all $x\neq 0$, $u(x,t)\to u^\ast (x)$ as $t\to T$ in $C^1$  of every compact  of $\R^N\setminus \{0\}$, with
\begin{equation}\label{prof_u}
u^\ast (x)\sim \Big [\frac{8p|\log|x||}{(p-1)^2|x|^2}\Big]^{\frac{1}{p-1}} ,\;\mbox{as}\; x\to 0,
\end{equation}
and for $|x|$ small,
\begin{equation}\label{prof_nabla}
\displaystyle |\nabla u^\ast(x)|\leq C |x|^{-\frac{p+1}{p-1}} \big|\log|x|\big|^{ \frac{p+1}{2(p-1)} -\alpha}.
\end{equation}

\end{enumerate}
\end{theo}

\begin{rem}
  \begin{enumerate}
  \item 
      With respect to our earlier work \cite{AZ1}
        and \cite{AZ2}, the improvement lays in \eqref{N_u}, 
        \eqref{N_nabla} and \textcolor{blue}{ \eqref{prof_nabla}}, since we could obtain the result in those
        papers only with $\epsilon_1=\frac 12$ and $\alpha=\frac{1}{4}$, whereas we are able
      now to get it for smaller values of $\epsilon_1$ and larger values of $\alpha$, namely for $\epsilon_1$ arbitrarily close to $0$ and $\alpha$ arbitrarily close to $\frac 12$. 
  \item
    As we have already noted in our earlier paper
      \cite{AZ1}, already when $\mu=0$, our error estimate in \eqref{N_u} is
      better than in \cite{MZ97}, thanks to the term
      $ 1+(\frac{|x|^2}{T-t})^{\frac{\beta}{2}}$ in the denominator. In fact, our estimate is even better than \eqref{b_u} proved in \cite{AZ1}, since we can take $\epsilon_1$ arbitrarily small and not just equal to $\frac 12$ as in \eqref{b_u}.

\item  As we have already noted in our previous paper
    \cite{AZ2}, the blow-up of the gradient at the origin is a
    consequence of the single-point blow-up of the solution and the
    mean value theorem. In this paper, we do better, by exhibiting
    some $\xi(t)\to 0$ as $t\to T$, so that $|\nabla(\xi(t),t)| \to
    \infty$ (see below in page \pageref{rk}). The improvement in \eqref{N_nabla} is crucial in
    obtaining that result, keeping in mind that the version with
    $\epsilon_1=\frac 12$ given in \cite{AZ1} (see \eqref{nabla_u}
    above) was not enough to derive such a sequence $\xi(t)$.
    \item We wonder whether one may be able to get an equivalent for the gradient final blow-up profile. Such a question seems to be hard, since, up to our knowledge, no such equivalent was proved in the literature, except in the recent work on the standard semilinear heat equation \eqref{equ} by Duong, Ghoul and Zaag \cite{DGZ21}.
    
\end{enumerate}
\end{rem}

The paper is organized as follows. In Section $2$, we give a formulation of the problem. In Section $3$, we prove  the existence of a solution of equation $(\ref{eq_v})$. Finally, in Section $4$,  we prove Theorem \ref{th1}.
\section{Formulation of the problem}
\textcolor{black}{In this section, we consider $T>0$ and $\eps_1\in (0, \frac{1}{2}]$, then 
we recall} the following similarity variables transformation  introduced by Giga and Kohn \cite{giga1}, \cite{giga2}, \cite{giga3}:
\begin{equation}\label{var_sim}
\displaystyle  y=\frac{x}{\sqrt{T-t}}, \;\;s=-\log(T-t)\;\;
\mbox{and}\;\;
 w(y,s)=\displaystyle (T-t)^{\frac{1}{p-1}}u(x,t),\end{equation}
where $T$ is the time where we want to make the solution blow up.\\
If $u(x,t)$ satisfies $(\ref{eq_u})$ for all $(x,t)\in \R^N\times [0, T)$, then $w(y,s)$ satisfies the following equation  for all $(y,s) \in \R^N\times[-\log T, +\infty)$:
\begin{equation}\label{eq_w}
\partial_s w=\Delta w-\frac12 y \cdot\nabla w-\frac{1}{p-1}w+|w|^{p-1}w+\mu e^{-\gamma s} |\nabla w|\int_{B(0, |y|)}|w|^{q-1},
\end{equation}
where $\gamma=\displaystyle \frac{p-q}{p-1}+\frac{N-1}{2}$.\\
Note from \eqref{hyp} that $\gamma>0$, which explains   the little effect of the perturbation term  for  large time $s$. 
This allows us to adopt a perturbative approach with respect to earlier constructions given for the standard semilinear heat equation \eqref{equ}.

\medskip

With the transformation \eqref{var_sim}, we reduce the construction of  a solution $u(x,t)$ for equation \eqref{eq_u} that blows up at $T<\infty$ 
to the construction of a global solution $w(y, s)$ for equation \eqref{eq_w}
such that 
\[
\forall s\ge s_0,\;\;
\epsilon_0 \le \|w(s)\|_{L^\infty}\le \frac 1{\epsilon_0},
\]
for some $\epsilon_0>0$. In fact, we will require more information on the solution, in the sense that 
 we would like to find $s_0>0$ and  initial data  $w_0$ such that the solution $w$ of equation (\ref{eq_w}) with $w(s_0)=w_0$ satisfies
$$\|w(y,s)-f(\frac{y}{\sqrt{s}})\|_{W^{1, \infty}_\beta}\to_{s\to \infty}0, $$
where $f$ is the profile defined by
\begin{equation}\label{profile}
\displaystyle f(z)=(p-1+\frac{(p-1)^2}{4p}|z|^2)^{-\frac{1}{p-1}}.
\end{equation}

As we have already pointed-out in \cite{AZ1}, we don't linearize \eqref{eq_w} around $f(\frac{y}{\sqrt{s}})+\frac{\kappa N}{2ps}$ as in the case of equation \eqref{equ} treated in \cite{MZ97}, since this function is not in the space $W^{1, \infty}_\beta$.  In fact,
we linearize equation $(\ref{eq_w})$ around a new profile
\begin{equation}\label{n_profile}
\displaystyle\varphi(y,s)=f(\frac{y}{\sqrt{s}})+\frac{\kappa N}{2ps}\chi_0(
\frac{y}{g_\eps(s)}),
\end{equation} where $\kappa=(p-1)^{-\frac{1}{p-1}}$ is a stationary solution for equation $(\ref{eq_w})$, $\chi_0\in C_0^\infty$ with ${\rm supp }(\chi_0)\subset B(0, 2)$,  $\chi_0\equiv 1$ on $B(0, 1)$ and $g_\eps(s) =s^{\frac{1}{2}+\eps}$, where $\eps$ is a fixed constant satisfying  \textcolor{black}{
\begin{equation}\label{choice}
0<\eps<\min(1, \eps_1 \beta^{-1}),
\end{equation}
(we could have taken $\eps =\frac{1}{2}\min(1,\eps_1 \beta^{-1} )$)}.\\
%
%
Introducing
\begin{equation}\label{v} v(y,s)=w(y,s)-\varphi(y,s),
\end{equation}
the problem is reduced to constructing a function $v$ such that 
$$ \lim_{s\to +\infty} \|v(s)\|_{W^{1, \infty}_\beta}=0.$$
If $w$ satisfies  equation $(\ref{eq_w})$, then $v$ satisfies the following equation
\begin{equation}\label{eq_v}
\partial_s v=(\mathcal{L}+V)v+B(v)+R(y,s)+\mathcal N(y, s),
\end{equation}
where
\begin{itemize}
\item the linear term  $(\mathcal{L}+V)v$  is given by 
\begin{equation}\label{linear_term}
\mathcal{L}(v)=\Delta v -\frac12 y\cdot\nabla v+v \; \mbox{ with }\;V(y, s)=p\varphi^{p-1}-\frac{p}{p-1},
\end{equation}

\item the nonlinear term is
\begin{equation}\label{nonlinear_term}
B(v)=|v+\varphi|^{p-1}(v+\varphi)-\varphi^{p}-p\varphi^{p-1}v,
\end{equation}
\item the remainder term  involving $\varphi$ is
\begin{equation}\label{rest_term}
R(y,s)=\Delta \varphi -\frac12 y\cdot\nabla \varphi -\frac{1}{p-1}\varphi+\varphi^p-\partial_s \varphi,
\end{equation}
\item and the ``new'' term is
\begin{equation}\label{N}
\displaystyle \mathcal{N}(y,s)=\mu e^{-\gamma s}|\nabla v+\nabla \varphi|\int_{B(0, |y|)} |v+\varphi|^{q-1}.
\end{equation}
\end{itemize}
Let us introduce the following integral form of equation $(\ref{eq_v})$, for each $s\geq \sigma\geq s_0$:
 \begin{equation}\label{eq_int_v}
v(s)=K(s, \sigma) v(\sigma)+\int_\sigma^s K(s,t)\big(B(v(t))+R(t)+\mathcal{N}(t)\big)dt,
\end{equation}
where $K$ is   the fundamental solution of the operator $\mathcal{L}+V$.\\
Since the dynamics of \eqref{eq_v} are influenced by the  linear operator  $\mathcal{L}+V$, we first need to  recall some  of its properties (for more details, see \cite{bricmont}).\\ The operator $\mathcal{L}$ is self-adjoint in $D(\mathcal{L})\subset L^2_\rho(\R^N)$, where $$L^2_\rho(\R^N)=\{v\in L^2_{loc}(\R^N);\quad \int_{\R^N}(v(y))^2 \rho(y) dy<\infty\}, \quad \rho(y)=\displaystyle \frac{e^{-\frac{|y|^2}{4}}}{(4\pi)^{\frac{N}{2}}}.$$
The spectrum of $\mathcal{L}$ consists only in eigenvalues  given by
$$ {\rm spec}(\mathcal{L})=\{1-\frac{m}{2}; \quad m\in \N\},$$
and its  eigenfunctions are rescaled Hermite polynomials.\\
If $N=1$, all the eigenvalues are simple, and the eigenfunction corresponding to $1-\frac{m}{2}$ is
\begin{equation}\label{vect_prop}
\displaystyle h_m(y)=\sum_{k=0}^{[\frac{m}{2}]}\frac{m!}{k! (m-2k)!}(-1)^ky^{m-2k}.
\end{equation}
In particular, $h_0(y)=1$, $h_1(y)=y$ and $h_2(y)=y^2-2$ are the eigenfunctions corresponding to nonnegative eigenvalues $\lambda \in\{1,\frac 12,0\}$. We remark that $h_m$ satisfies:
$$\displaystyle \int_{\R} h_n h_m \rho dx=2^nn! \delta_{n,m}.$$
We also introduce $\displaystyle k_m=\frac{h_m}{\|h_m\|_{L^2_\rho(\R)}^2}$.\\
If $N\geq 2$, the eigenspace corresponding to $1-\frac{m}{2}$ is given by
$$E_m=\{h_{m_1}(y_1)\cdots h_{m_N}(y_N); \quad \quad m_1+\cdots m_N=m\}.$$
In particular, these are the eigenspaces corresponding to nonnegative eigenvalues:
$$E_0=\{1\}, \; E_1=\{y_i;\quad  \; i=1,\cdots, N\}\; \mbox{ and} \;E_2=\{h_2(y_i),\;  y_iy_j;\quad   i,j=1, \cdots , N,\;  i\neq j\}.$$
The potential $V(y,s)$ has two fundamental properties:
\begin{itemize}
\item $V(., s)\to 0$ in $L^2_\rho$ as $s\to +\infty$. In particular, on bounded sets or in the ``blow-up area'' $\{|y|\leq K_0\sqrt{s}\}$, where $K_0>0$ will be fixed large enough, the effect of $V$ is considered as a perturbation of the effect of $\mathcal{L}$, except maybe on the null eigenspace of $\mathcal{L}$.
\item Outside the  ``blow-up  area'', we have the following property: for all $\eta>0$ the exist $C_\eta>0$ and $s_\eta$ such that
$$\displaystyle \sup_{s\geq s_\eta , |y|\geq C_\eps \sqrt{s}}|V(y,s)-(-\frac{p}{p-1})|\leq \eta.$$
\end{itemize}

This means that $\mathcal{L}+V$ behaves like $\mathcal{L}-\frac{p}{p-1}$ in the region $\{|y|\geq K_0\sqrt{s}\}$ for $K_0$ large enough. Since $-\frac{p}{p-1}<-1$ and $1$ is the largest eigenvalue of $\mathcal{L}$, we can consider the operator
$\mathcal{L}+V$ 
outside the blow-up area as an operator with a purely negative spectrum, which  greatly simplifies the analysis.\\
Bearing in mind 
that the behavior of $V$ is not the same inside and outside the  ``blow-up''  area, we decompose $v$ as follows. Let us first introduce the following cut-off function:
\begin{equation}\label{chi}
\chi(y, s) =\chi_0(\displaystyle \frac{|y|}{K_0\sqrt{s}}),
\end{equation}  where $K_0>0$ 
will 
be fixed large enough so that various technical estimates hold, and the cut-off function $\chi_0$ was already introduced after \eqref{n_profile}.\\
We write
\begin{equation}\label{decomp_v}
v(y, s)= v_b(y,s)+v_e(y,s),
\end{equation}
with
\begin{equation}\label{decomp_v_b}
v_b(y,s)=v(y,s)\chi(y,s), \quad v_e(y,s)=v(y,s)(1-\chi(y,s)).
\end{equation}
We note that ${\rm supp }\; v_b(s)\subset B(0, 2K_0\sqrt{s})$ and ${\rm supp }\; v_e(s)\subset \R^N\backslash  B(0, K_0\sqrt{s})$.\\
In order to control $v_b$, we decompose it according  to the sign of the eigenvalues of $\mathcal{L}$ as follows (for simplicity, we give the decomposition only for $N=1$, bearing in mind that the situation for $N\ge 2$ is the same, except for some more complicated notations that can be found in Nguyen and Zaag \cite{NZ0}, where the formalism in higher dimensions is extensively given):
\begin{equation}\label{decomp}
\displaystyle v(y, s)= v_b(y,s)+v_e(y,s) =\sum_{m=0}^{2}v_m(s)h_m(y)+v_{-}(y, s)+v_e(y,s),
\end{equation}
where for $0\leq m\leq 2$, $v_m=P_m(v_b)$ and $v_{-}(s)=P_{-}(v_b)$, with $P_m$ being the $L^2_\rho$ projector on $h_m$, the eigenfunction corresponding to $\lambda=1-\frac{m}{2}$, and $P_{-}$ the projector on $\{h_i\;| \; i\geq 3\}$, the nonpositive subspace of the operator $\mathcal{L}$.

\section{Existence}
This  section  is devoted to the proof of the existence of a solution $v$ of $(\ref{eq_v})$ such that 
\begin{equation}\label{goal}
\displaystyle \lim_{s\to +\infty}\|v(s)\|_{W^{1, \infty}_\beta}=0.
\end{equation}
We proceed exactly as in our previous paper \cite{AZ1}, except for the definition of the shrinking set below in Definition \ref{def_shrinking}, where we will introduce a smaller set, leading to a sharper estimate on the norm in \eqref{goal}, at the expense of more involved estimates.
For simplicity in the notations, we give the proof only when  $N=1$.

\medskip

As in our previous paper \cite{AZ1},
our argument is a non trivial adaptation of the method performed of equation \eqref{equ} by  Bricmont and Kupiainen
 \cite{bricmont}, Merle and Zaag \cite{MZ97}, Nguyen and Zaag \cite{NZ2} and Tayachi and Zaag  \cite{TZ}. 
 Accordingly, as in \cite{AZ1}, we proceed in 4 steps:
\begin{itemize}
\item In the first step, we define a  new shrinking set $\mathcal{V}_{A, \eps_1}(s)$ and translate our goal in \eqref{goal} to the construction of a solution belonging to that set.

\item In the second step, using the spectral properties, we reduce
our goal from the control of $v(s)$ (an infinite dimensional variable) in $\mathcal{V}_{A, \eps_1}(s)$ to
the control of its two first components $(v_0(s), v_1(s))$  in $ [-\frac{A}{s^2},\frac{A}{s^2}]^2$ (a two-dimensional variable).
\item In the third step,  we prove a parabolic regularity estimate, in order to obtain  the convergence of the gradient term.
\item In the last step, we solve the two-dimensional problem using index theory
and  we prove the existence of a solution  $v$ of \eqref{eq_v} belonging to the new  shrinking set $\mathcal{V}_{A, \eps_1}(s)$.
\end{itemize} 

In the following, we fix $\beta$ satisfying \eqref{beta} and  we denote by $C$ a generic positive constant, depending only on $p$, $\mu$, $\beta$ and $K_0$. Note that $C$  does not depend on $A$,  $\eps_1$  and $s_0$, the constants that will appear below.
\subsection{A New  shrinking set $\mathcal{V}_{\beta, K_0, A,  \eps_1}(s)$ }
In this subsection, 
 we introduce  the  new shrinking set,
 which is the main novelty with respect to our previous work \cite{AZ1}:
\begin{defi}[A set shriking to zero] \label{def_shrinking} For all  $K_0\ge 1$, $A\geq 1$, $0<\eps_1\le \frac{1}{2}$ and  $s\geq 1$, we define $\mathcal{V}_{\beta, K_0, A,  \eps_1}(s)$ (or $\mathcal{V}_{A, \eps_1}(s)$ for simplicity) as the set of all functions  $g$ such that $(1+|y|^{\beta})g\in L^\infty(\R^N)$  and
\begin{eqnarray}\label{shri_1}
|g_k|\leq \displaystyle \frac{A}{s^2},\; k=0, 1, \;\;  |g_2|\leq \displaystyle \frac{A^2 \log s}{s^2}, \;\; \|\displaystyle \frac{g_-}{1+|y|^3} \|_{L^\infty}\leq  \displaystyle \frac{A}{s^{\frac{5}{2}-\eps_1}},
\end{eqnarray}
\begin{eqnarray}\label{shri_2}
 \|\displaystyle g_e \|_{L^\infty}\leq  \displaystyle \frac{A^2}{s^{1-\eps_1}}
, \;\;   \|\displaystyle (1+|y|^\beta) g_e \|_{L^\infty}\leq  \displaystyle \frac{A^2}{s^{1-\frac{\beta}{2}-\eps_1}},
\end{eqnarray}
where $g_k$, $g_-$ and $g_e$ are defined in \eqref{decomp_v_b} and \eqref{decomp}.
\end{defi}
\begin{rem}\label{remshrin}
\begin{enumerate}
\item  Since $p>3$ and $\beta$ satisfy \eqref{beta}, it follows  that $\beta <\frac{2}{p-1}<1$. Therefore, since $\epsilon_1\in(0, 1- \frac \beta 2)$, it follows that the bounds on $g_-$ and $g_e$ go to zero as $s\to \infty$. 
\item When $\epsilon_1=\frac 12$, we recover the shrinking set of our previous paper \cite{AZ1}. Since we will take $\epsilon_1$ arbitrarily small here, we clearly improve our estimates. This is the key novelty of our work.

\end{enumerate}
\end{rem}
Since  this shrinking set is different from all the previous studies, we need to introduce new estimates.
Let us first note that
the set  $\mathcal{V}_{A, \eps_1}(s)$ is increasing (for fixed $s,\; \beta,\; K_0, \;  \eps_1$) with respect to $A$ in the sense of inclusion.
Let us now give some properties of the shrinking set in the following:
\begin{pro}\label{prop_shrin}
For all $K_0\ge 1$, $A\geq 1$ and $0<\eps_1\le \frac{1}{2}$, there exists $s_1(K_0, A, \eps_1)$ such that, for all $s\geq s_1(K_0, A, \eps_1)$ and $g\in \mathcal{V}_{A, \eps_1}(s)$, we have
\begin{itemize}
\item[i)] \begin{equation}\label{est_shrin_1}
\|g\|_{L^\infty(|y|\leq 2K_0\sqrt{s})}\leq \displaystyle \frac{CA}{s^{1-\eps_1}}\quad \mbox{and}\quad \|g\|_{L^\infty(\R)} \leq \displaystyle \frac{CA^2}{s^{1-\eps_1}}.
\end{equation}
\item[ii)] \begin{equation}\label{est_shrin}
\|(1+|y|^\beta)g\|_{L^\infty(|y|\leq 2K_0\sqrt{s})}\leq \displaystyle \frac{CA}{s^{1-\frac{\beta}{2}-\eps_1}}\;\; \mbox{and}\;\; \|(1+|y|^\beta)g\|_{L^\infty(\R)} \leq \displaystyle \frac{CA^2}{s^{1-\frac{\beta}{2}-\eps_1}}.
\end{equation}
\end{itemize}
\end{pro}
\begin{rem}
Here and throughout the paper, the generic constant $C$ depends on $p$ and $N$, and may depend on $K_0$ too.
\end{rem}
We omit the proof of this proposition, since it is  obtained by a simple modification of the proof of    Proposition $3.8$ in \cite{AZ1} page $14$.

\medskip

 The construction of a solution $v$ of \eqref{eq_v} is based on a careful choice of initial data  at time $s_0=- \log T$, as we do in the following:
\begin{defi}(Choice of the initial data) For all $K_0\ge 1$,  $A\geq 1$, $s_0=-\log T>1$ and $d_0, \, d_1\in \R$, we consider the following function as initial data for equation $(\ref{eq_v})$:
\begin{equation}\label{ci}
\psi_{s_0, d_0, d_1}(y)=\displaystyle \frac{A}{s_0^2}(d_0h_0(y)+d_1h_1(y))\chi(2y, s_0),
\end{equation}
where $h_i$, $i=0,1$ are defined in $(\ref{vect_prop})$ and $\chi$ is defined in $(\ref{chi})$.
\end{defi}
Reasonably, we choose  the parameter $(d_0, d_1)$ such that  the initial data $\psi_{s_0, d_0, d_1}\in \mathcal{V}_{A, \eps_1}(s_0)$. More precisely, we claim the following result:
\begin{pro}\label{prop_ci}
(Properties of initial data) For each $0<\eps_1\le \frac{1}{2}$,  $K_0\ge 1$, $A\geq 1$, there exists $s_{01}(K_0, A, \eps_1)>1$ such that for all $s_0\geq s_{01}(K_0, A, \eps_1)$:
\begin{itemize}
\item[i)] There exists a rectangle $\mathcal{D}_{s_0}\subset[-2, 2]^2$ such that the mapping
\begin{equation}
\begin{array}{lcl}

\Phi : \R^2&\to &\R^2\\
(d_0, d_1)&\mapsto&(\psi_0, \psi_1),
\end{array}
\end{equation}
(where $\psi:=\psi_{s_0, d_0, d_1}$) is linear, one to one from $\mathcal{D}_{s_0}$ onto $[\displaystyle -\frac{A}{s_0^2},\displaystyle \frac{A}{s_0^2}]^2$ and maps  $\partial\mathcal{D}_{s_0}$ into $\partial([\displaystyle -\frac{A}{s_0^2},\displaystyle \frac{A}{s_0^2}]^2)$. Moreover, it has degree one on the boundary.\\
\item[ii)] For all $(d_0, d_1)\in \mathcal{D}_{s_0} $, $\psi \in \mathcal{V}_{A, \eps_1}(s_0)$ with strict inequalities except for $(\psi_0, \psi_1)$, in the sense that
\begin{equation}\label{psi_initial}
\psi_e \equiv 0,\quad  |\psi_-(y)|< \frac{1}{s_0^{\frac{5}{2}-\eps_1}}(1+|y|^3),\;  \forall y\in \R,\\
|\psi_k|\leq \frac{A}{s_0^2},\;  k=0,1, \quad |\psi_2|< \displaystyle \frac{ \log s_0}{s_0^2}.
\end{equation}
\item[iii)] Moreover,  for all $(d_0, d_1)\in  \mathcal{D}_{s_0}$, we have
\begin{eqnarray}
\|(1+|y|^\beta)\nabla \psi\|_{L^\infty}&\leq& \frac{CA}{s_0^{2-\frac{\beta}{2}}}\leq \frac{1}{s_0^{1-\frac{\beta}{2}-\eps_1}},\label{n_ci}\\
|\nabla \psi_-(y)| &\leq& \frac{1}{s_0^{\frac{5}{2}-\eps_1}}(1+|y|^3).\label{ci_-}
\end{eqnarray}
\end{itemize}
\end{pro}

We omit the proof, since it is obtained by a simple modification of the proof of Proposition $4.5$ of \cite{TZ}.
For more details,  we refer the interested reader  to pages $5915-5918$ of \cite{TZ}, and pages $14-15$ of  \cite{AZ1}.

\subsection{Parabolic Regularity}
 As stated above, our goal is to get   the convergence in $W^{1, \infty}_\beta$. Thus,  we need to prove that the solution $v$ of \eqref{eq_v} satisfies 
 $$\displaystyle \lim_{s\to +\infty}\|(1+|y|^\beta) \nabla v(s)\|_{L^\infty}=0.$$
 More precisely, we need the following parabolic regularity  result  for equation $(\ref{eq_v})$:
\begin{pro}\label{prop-reg}(Parabolic regularity in $\mathcal{V}_{A, \eps_1}(s)$)\\
For all $0<\eps_1\le \frac{1}{2}$,  $K_0\ge 1$ and  $A\geq 1$, there exists  $s_{02}(K_0, A, \eps_1)$ such that for all $s_0\geq s_{02}(K_0,A)$, if $v$ is  the solution of $(\ref{eq_v})$  for all $s\in [s_0, s_1]$, $s_0\leq s_1$, with initial data at $s_0$, given in $(\ref{ci})$ with  $(d_0, d_1)\in \mathcal{D}_{s_0}$ defined in Proposition \ref{prop_ci}, and  $v(s)\in \mathcal{V}_{A, \eps_1}(s)$, then,  for all $s\in[s_0, s_1]$, we have
\begin{eqnarray}\label{reg}
\|\displaystyle \nabla v(s) \|_{L^\infty}\leq  \displaystyle \frac{CA^2}{s^{1-\eps_1}} \quad \mbox{and} \quad \|\displaystyle (1+|y|^\beta)\nabla v(s) \|_{L^\infty}\leq  \displaystyle \frac{CA^2}{s^{1-\frac{\beta}{2}-\eps_1}}.
\end{eqnarray}
\end{pro}
The proof of the previous proposition  follows exactly as in \cite{AZ1}. For that reason, the proof is omitted and we refer the reader to  subsection 3.2.3 page 21-24 in \cite{AZ1} for details.\\

Since $v(s)\in \mathcal{V}_{A, \eps_1}(s)$, we clearly see from $(\ref{est_shrin})$ and  $(\ref{reg})$ that
\begin{equation}\label{double}
\|\displaystyle (1+|y|^\beta) v(s) \|_{L^\infty}+\|\displaystyle (1+|y|^\beta)\nabla v(s) \|_{L^\infty}\leq  \displaystyle \frac{CA^2}{s^{1-\frac{\beta}{2}-\eps_1}}.
\end{equation}

\subsection{Reduction to a finite dimensional problem}\label{sub_red}
The following section is crucial in the proof of the existence of the blow-up solution. In fact, we reduce here the 
construction question
to a finite dimensional problem. As in \cite{MZ}, \cite{EZ},  \cite{TZ} and \cite{AZ1}, we prove that it is enough to control the $2$ components $(v_0(s), v_1(s))\in [\displaystyle -\frac{A}{s^2},\displaystyle \frac{A}{s^2}]^2$ in order to control the solution $v(s)$ in $\mathcal{V}_{A, \eps_1}(s) $,  which is infinite dimensional. Precisely, we have the following proposition.
\begin{pro}\label{prop_red}(Control of $v(s)$ by $(v_0, v_1)(s)$ in $\mathcal{V}_{A, \eps_1}(s)$ ) For all $0<\eps_1\le \frac{1}{2}$, there exists $K_3(\eps_1)\ge 1$ such that for any $K_0\ge K_3(\eps_1)$, there exists $A_3(K_0, \eps_1) \geq 1$ such that for each
$A \geq A_3(K_0, \eps_1)$, there exists $s_{03}(K_0, A, \eps_1)\in\R$ such that for all $s_0\geq s_{03}(K_0, A, \eps_1)$, the following holds:\\
If $v$ is a solution of $(\ref{eq_v})$ with initial data at $s = s_0$ given by $(\ref{ci})$ with $(d_0, d_1) \in  \mathcal{D}_{s_0}$, and
$v(s)\in \mathcal{V}_{A, \eps_1}(s)$ for all $s \in  [s_0, s_1]$, with $v(s_1) \in  \partial \mathcal{V}_{A, \eps_1}(s_1)$ for some $s_1 \geq s_0$, then:
\begin{itemize}
\item[i)](Reduction to a finite dimensional problem) We have:
$$(v_0(s_1), v_1(s_1))\in \partial([\displaystyle -\frac{A}{s_1^2},\displaystyle \frac{A}{s_1^2}]^2).$$
\item[ii)](Transverse crossing) There exist $m \in \{0, 1\}$ and $\omega\in \{-1, 1\}$ such that
$$\omega v_m(s_1) =\displaystyle \frac{A}{s_1^2}\quad \mbox{and}\quad
\omega v'_m(s_1) > 0.$$
\end{itemize}
\end{pro}

The remainder of this section is devoted to the  proof of  Proposition \ref{prop_red}. We proceed in $2$ steps:
\begin{itemize}
\item In the first step, we prove that if  $v(s)\in \mathcal{V}_{A, \eps_1}(s)$, then  $B(v)$, $R(y,s)$ and $\mathcal{N}(y,s)$ given in $(\ref{eq_v})$ belong
$\mathcal{V}_{C, \eps_1}(s)$ and the potential term $Vv(s)\in  \mathcal{V}_{CA, \eps_1}(s)$, for some positive constant $C$.
\item In the second  step, we prove the result of the reduction to a finite dimensional problem.
\end{itemize}
\subsubsection{Step 1: Preliminary estimates on various terms of equation $(\ref{eq_v})$}
In the following, we prove  that   the 
remainder
term $R(y, s)$ is  
in $\mathcal{V}_{ C, \eps_1}(s)$ for $s$ large enough and some $C>0$. We  also prove that  if $v(s)\in  \mathcal{V}_{A, \eps_1}(s)$, then  the nonlinear term $B(v)\in \mathcal{V}_{ C, \eps_1}(s)$ and the potential term $Vv$ is trapped  in $\mathcal{V}_{ CA, \eps_1}(s)$. Furthermore, under some additional assumptions on $v$, we prove that the new term $\mathcal{N}(y,s)$ is also  trapped in $\mathcal{V}_{ C, \eps_1}(s)$.
More precisely, we prove the following result:
\begin{lem}\label{lem_est}
\begin{enumerate}
\item For all $K_0\ge 1$ and $0<\eps_1\le \frac{1}{2}$, there exists $s_3(K_0, \eps_1)$ sufficiently large such that for all $s\geq s_3(K_0, \eps_1)$,  the 
remainder
term
 $R(s)\in \mathcal{V}_{ C, \eps_1}(s)$.
\item For all $K_0\ge 1$, $0<\eps_1\le \frac{1}{2}$ and $A\geq 1$, there exists $s_4(K_0, A, \eps_1)$ sufficiently large such that for all $s\geq s_4(K_0, A, \eps_1)$, if $v(s)\in  \mathcal{V}_{A, \eps_1}(s)$, then  the nonlinear term $B(v)\in \mathcal{V}_{C, \eps_1}(s)$ and the potential term $Vv \in\mathcal{V}_{CA, \eps_1}(s)$.
\end{enumerate}
\end{lem}
We need the following technical lemma before proving Lemma  \ref{lem_est}.

\begin{lem} \label{lem_tech}
For all $K_0\ge 1$, there  exists $s_2(K_0)$ such that for all $s\geq s_2(K_0)$,  the following holds:\\
(i) if $g(y)=1$ then $\left\|\displaystyle \frac{g_-(y)}{1+|y|^3}\right\|_{L^\infty}\leq \frac{C}{K_0^3s^{\frac{3}{2}}}$,\\
(ii) if $g(y)=y^2$, then $\left\|\displaystyle \frac{g_-(y)}{1+|y|^3}\right\|_{L^\infty}\leq \frac{C}{K_0\sqrt s}$.
\end{lem}
\begin{proof}
The proof of this technical lemma follows exactly as in \cite{TZ}. For that  reason, the proof  is omitted,  and we refer the  interested  reader to Lemma $4.8$ page $5916-5917$ in \cite{TZ} for a similar case.
\end{proof}
Now, we are ready to prove Lemma \ref{lem_est}.


\begin{proof}[Proof of Lemma \ref{lem_est}]
In comparison with \cite{AZ1}, the estimates for $B(v)$ and $vV$ can be adapted smoothly (see subsection 3.2.2 page $ 14-17$ in \cite{AZ1} and subsection 4.2.2  page $5918-5923$ in  \cite{TZ}), unlike for $R(y,s)$. For that reason, we only prove the estimate for the latter.
Moreover, since the constraints on $R_m(s)$ for $m=0,1,2$ in our new shrinking set are the same as with the old shrinking set in \cite{AZ1} (see Remark \ref{remshrin}), we refer the reader to Lemma 3.9 page 15 in that paper for the proof of the estimates on those components, and focus only on the proofs of the estimates related to $R_-(s)$ and $R_e(s)$.\\
Since $-\frac{1}{2}z f'(z)-\frac{1}{p-1}f(z)+f(z)^p=0$ by definition \eqref{profile} of $f(z)$, we write $R(y, s)$ introduced in \eqref{rest_term} as follows
\begin{eqnarray}\label{R}
R(y,s)&=&\displaystyle \frac{1}{s} f''(z)+\frac{1}{2s}z f'(z)+ \big(f(z)+\chi_0(Z)\frac{\kappa}{2ps}\big)^p-f(z)^p\nonumber \\
&+& \frac{\kappa}{2ps}\Big[ \frac{1}{g_\epsilon^2}\chi_0''(Z)-\big(\frac{1}{2}-\frac{g_\epsilon'(s)}{g_\epsilon(s)}\big)Z\chi_0'(Z)+(\frac{1}{s}-\frac{1}{p-1})\chi_0(Z)\Big]\nonumber\\
&=& R_i+R_{ii},
\end{eqnarray}
where $\displaystyle z=\frac{y}{\sqrt{s}}, \;Z=\frac{y}{g_\epsilon(s)}$, $g_\epsilon(s)=s^{\frac{1}{2}+\eps}$, and $\eps\in (0, \min(1, \textcolor{black}{\eps_1\beta^{-1}}))$ was already fixed in \eqref{choice}.\\
First, noting that
$f(z)$ 
is
bounded by definition \eqref{profile}, then
recalling that $p>3$ (see \eqref{hyp}) and $0\le \chi_0(Z)\le 1$ (see right after \eqref{n_profile}), we write
\begin{equation}\label{taylor1}
|\big(f(z)+\chi_0(Z)\frac{\kappa}{2ps}\big)^p-f(z)^p|\le \frac Cs.
\end{equation}
Noting that $f''(z)$ and $zf'(z)$ are also bounded, still by \eqref{profile}, then recalling that $\chi_0$ is smooth and compactly supported, we clearly see from \eqref{R} that
\begin{equation}\label{RLinf}
|R(y, s)|\leq \frac{C}{s}.
\end{equation}
In particular, 
by definition \eqref{decomp_v_b}, we have for $s$ large enough:
$$|R_e(y, s)|\leq \frac{C}{s}\leq \frac C{s^{1-\eps_1}}.$$
  
 Now, we estimate $\|(1+|y|^{\beta})R_e(y, s)\|_{L^\infty}$. Since we have the 
 same
 profile as in \cite{AZ1},   arguing as  in pages 15-17 of \cite{AZ1}, we 
 easily
 prove
 the following estimate  
%
$$\|(1+|y|^{\beta})R_e(y,s)\|_{L^\infty}\leq \frac{C}{s^{1-\frac{\beta}{2}-\eps \beta}}.$$
Since $\eps \beta\leq \eps_1$ by \eqref{choice},
we obtain  for $s$ large enough
$$\|(1+|y|^{\beta})R_e(y,s)\|_{L^\infty}\leq \frac{C}{s^{1-\frac{\beta}{2}-\eps_1}}.$$
Finally, we 
bound
$R_-(y,s)$.\\
Refining the Taylor expansion \eqref{taylor1} up to the second order, we write
\begin{eqnarray}\label{Ri}
R_i= \frac 1s( f"(z)+\frac{z}{2}f'(z)+\frac{\kappa}{2p}\chi_0(Z)f(z)^{p-1})
+\frac{p(p-1)}{2}(\frac{\kappa}{2ps})^2 \chi_0^2(Z)f^{p-2}(z)+O\left(\frac{1}{s^3}\right).
\end{eqnarray}
Since $\chi_0$ is constant on the unit ball $B(0,1)$ (see right after \eqref{n_profile}), making a Taylor expansion of $R_i$ \eqref{Ri}
in $z$ and in $Z$, we see that
\[
R_i = \frac 1s \left(a+bz^2+O(z^3)\right) +\frac c{s^2} (1+O(z^2))+O(\frac 1{s^3}) +O\left(\frac{|Z|^3}s\right)
\]
for some real numbers $a$, $b$ and $c$.
Similarly, we have from \eqref{R}
\[
R_{ii} = \frac \kappa{2ps}\left(\frac 1s - \frac 1{p-1}\right)+O\left(\frac{|Z|^3}s\right).
\]
Using Lemma \ref{lem_tech} together with straightforward estimates, we see that
\begin{equation*}
|R_-|\leq  \frac{C}{s^{\frac{5}{2}}}(1+|y|^3)\leq  \frac{C}{s^{\frac{5}{2}-\eps_1}}(1+|y|^3)
\end{equation*}
for $s$ large enough.
This  concludes the proof of Lemma \ref{lem_est}.
\end{proof}
We now estimate the new term defined in \eqref{N}. For this end, we claim first, the  following Lemma:
\begin{lem}\label{lem_new_term} For all $0<\eps_1\le \frac{1}{2}$, $K_0\ge 1$ and  $A\geq 1$, there exists $s_5(K_0, A, \eps_1)$ sufficiently large, such that for all $s\geq s_5(K_0, A, \eps_1)$, if $v\in \mathcal{V}_{A, \eps_1}(s)$
satisfies the following 
\begin{equation}\label{reg_v}
\|\nabla v(s)\|_{L^\infty}\leq \frac{CA^2}{s^\tau}, \quad  \|(1+|y|^{\beta})\nabla v(s)\|_{L^\infty}\leq \frac{CA^2}{s^{\tau'}},
\end{equation} 
for some positive constants  $\tau$ and $\tau'$,
then, the new term defined in \eqref{N} satisfies
\begin{enumerate}
\item $\|\mathcal{N}(y,s)\|_{L^\infty}\leq Ce^{-\frac{\gamma}{2}s}$,
\item $\|(1+|y|^{\beta})\mathcal{N}_e(y,s)\|_{L^\infty}\leq Ce^{-\frac{\gamma}{2}s}$,
\end{enumerate}
where $C$ is a positive constant.
\end{lem}
\begin{proof}
Since $V_{A,\epsilon_1}$ is increasing (in the sense of inclusion) with respect to $\epsilon_1$, recalling that $0<\epsilon_1\le \frac 12$ and that the case $\epsilon_1=\frac 12$ was treated in Abdelhedi and Zaag \cite{AZ1}, we refer the read to Lemma 3.11 page 18 of \cite{AZ1}.\\
\end{proof}
Thanks to  Lemma \ref{lem_new_term}, we obtain the following results:

\begin{pro}\label{prop_new_term} For all $0<\eps_1\le \frac{1}{2}$, $K_0\ge 1$ and  $A\geq 1$, there exists $s_6(K_0, A, \eps_1)$ sufficiently large, such that for all $s\geq s_6(K_0, A, \eps_1 )$, if $v\in \mathcal{V}_{A, \eps_1}(s)$  and 
$\nabla v$ satisfies the estimate stated in Proposition \ref{prop-reg}, then, the new term $\mathcal{N}$ defined in \eqref{N} satisfies
$$|\mathcal{N}_m(s)|\le Ce^{-\frac \gamma2 s}, \,0\le m\le 2, \quad \displaystyle \|\frac{\mathcal{N}_{-}(y,s)}{1+|y|^3}\|_{L^\infty}\leq  Ce^{-\frac \gamma2 s} $$ and  $\mathcal{N}\in\mathcal{V}_{C, \eps_1}(s)$, for some positive constant $C$.
\end{pro}
\begin{proof}
The proof follows by standard estimates exactly as in the case $\epsilon_1=\frac 12$ treated in \cite{AZ1}. See Proposition 3.10 page 18 in that paper.
\end{proof}

\subsubsection{ Step 2: Reduction to a finite dimensional problem}
 In this subsection, we reduce the problem to a finite dimensional one. We prove through a priori estimates that the control of $v(s)$ in $\mathcal{V}_{A, \eps_1}(s)$ 
 can be
reduced to the control of $(v_0, v_1)(s)$ in $[-\frac{A}{s^2}, \frac{A}{s^2}]^2$.\\
 For this end, we 
 recall
 the integral equation \eqref{eq_int_v}, for each $s\geq \sigma \geq s_0$ :
\begin{eqnarray}\label{decomp1}
v(s)&=&K(s, \sigma) v(\sigma)+\int_\sigma^s K(s,t)(B(v(t))+R(t)+\mathcal{N}(t))dt\nonumber\\
&=& \mathcal{A}(s)+\mathcal{B}(s)+\mathcal{C}(s)+\mathcal{D}(s),
\end{eqnarray}
and we  estimate  the different components of this Duhamel formulation.\\
In the first step, 
we recall from \cite{AZ1} this fundamental property concerning the kernel $K(s, \sigma)$:
 \begin{lem}
   \label{lem_14}
  There exists $K_5\geq 1$ such that for all $K_0\ge K_5$, for all $\rho >0$, there exists $\sigma_0=\sigma_0(K_0,\rho)$ such that if $\sigma\geq \sigma_0\geq 1$ and $g(\sigma)$ satisfies
\begin{equation}\label{bk}
  \displaystyle \sum_{m=0}^2 |g_m(\sigma)|+\|\frac{g_{-}(y,
    \sigma)}{1+|y|^3}\|_{L^\infty}+\|(1+|y|^\beta)g_e(\sigma)\|_{L^\infty}<+\infty,
  \end{equation}
 then $\theta(s)=K(s, \sigma)g(\sigma)$ satisfies for all $s\in [\sigma,\sigma+\rho]$,
 \begin{enumerate}
 \item \begin{eqnarray*}\displaystyle \|\frac{\theta_-(y,s)}{1+|y|^3}\|_{L^\infty}&\leq &C\frac{e^{s-\sigma}((s-\sigma)^2+1)}{s}(|g_0(\sigma)|+|g_1(\sigma)|+\sqrt{s}|g_2(\sigma)|)\\&&+Ce^{-\frac{s-\sigma}{2}}\|\frac{g_-(y,\sigma)}{1+|y|^3}\|_{L^\infty} +C\frac{e^{-(s-\sigma)^2}}{s^{\frac{3}{2}}}\|g_e(\sigma)\|_{L^\infty}
 \end{eqnarray*}
 \item  $|\theta_e(s)|\leq \displaystyle Ce^{s-\sigma}\big(\sum_{l=0}^2 s^{\frac{l}{2}}|g_l(\sigma)|+s^{\frac{3}{2}}\|\frac{g_{-}(y, \sigma)}{1+|y|^3}\|_{L^\infty}\big)+Ce^{-\frac{s-\sigma}{p}}\|g_e(\sigma)\|_{L^\infty}.$
 \item \begin{eqnarray*}\displaystyle \|(1+|y|^{\beta})\theta_e(y,s)\|_{L^\infty}&\leq &\displaystyle Ce^{-\frac{s-\sigma}{2}(\frac{1}{p-1}-\frac{\beta}{2})}\|(1+|y|^{\beta})g_e(y,\sigma)\|_{L^\infty}\\
&&+ \displaystyle Ce^{\frac{\beta}{2}(s-\sigma)}s^{\frac{\beta}{2}}(\sum_{l=0}^2 s^{\frac{l}{2}}|g_l(\sigma)| +s^{\frac{3}{2}}\|\frac{g_-(y,\sigma)}{1+|y|^3}\|_{L^\infty}).
 \end{eqnarray*}
 \end{enumerate}
 \end{lem}
\begin{proof}
The original version of this lemma  is due to
     Bricmont and Kupiainen in \cite{bricmont}, where they gave the proof in the case where some particular bounds hold on each component appearing in \ref{bk}, only for $\beta=0$.
     Later in \cite{NZ2}, Nguyen and Zaag simply
   restated the lemma of Bricmont and Kupiainen and checked that the proof of \cite{bricmont}
   works without those particular bounds, still for $\beta=0$.
   In our earlier paper \cite{AZ1}, we handled the case $0\le \beta <\frac 2{p-1}$.
   Accordingly, we refer the
 reader to Lemma 3.15 page 26 in \cite{AZ1} for the detailed justification.
 \end{proof}

 \medskip
 
 Applying the above Lemma, we get  a new bound  on  all terms in the decomposition $(\ref{decomp1})$. More precisely, we have the following results:
\begin{lem}\label{lem15} For all $0<\eps_1\le \frac{1}{2}$, there exists $K_6(\eps_1)$ such that for any $K_0\ge K_6(\eps_1)$, there exists  $A_6(K_0, \eps_1)>0$ such that for all $A\geq A_6(K_0, \eps_1)$ and $\rho>0$, there exists $s_{06}(K_0, A, \rho, \eps_1)>0$ with the following property: for all $s_0\geq s_{06}(K_0, A, \rho, \eps_1)$ assume that for all $s\in [\sigma, \sigma+\rho]$, $v(s)$ satisfies $(\ref{eq_v})$, $v(s)\in \mathcal{V}_{A, \eps_1}(s)$ and $\nabla v$  satisfies the estimates stated in Proposition \ref{prop-reg}.
Then, we have:
\begin{enumerate}
\item Linear term:
$$\|\frac{\mathcal{A}-(y,s)}{1+|y|^3}\|_{L^\infty}\leq Ce^{-\frac{1}{2}(s-\sigma)}\|\frac{v_-(y,\sigma)}{1+|y|^3}\|_{L^\infty}+C\frac{e^{-(s-\sigma)^2}}{s^{\frac{3}{2}}}\|v_e(\sigma)\|_{L^\infty} +\frac{C }{s^{\frac{5}{2}-\eps_1}}. $$

$$\|\mathcal{A}_e\|_{L^\infty}\leq Ce^{-\frac{s-\sigma}{p}}\|v_e(\sigma)\|_{L^\infty}+ Ce^{s-\sigma}s^{\frac{3}{2}}\|\frac{v_-(y,\sigma)}{1+|y|^3}\|_{L^\infty}+\frac{C }{s^{1-\eps_1}}.$$
 \begin{eqnarray*}
\|(1+|y|^{\beta})\mathcal{A}_e(s)\|_{L^\infty}\!\!\!\!\!\!&\leq&\!\!\! Ce^{-\frac{s-\sigma}{2}(\frac{1}{p-1}-\frac{\beta}{2})}\| (1+|y|^{\beta})v_e(\sigma)\|_{L^\infty}\!+\! Ce^{\frac{\beta}{2}(s-\sigma)}s^{\frac{\beta}{2}+\frac{3}{2}}\|\frac{v_-(y,s)}{1+|y|^3}\|_{L^\infty}\\&+&\frac{C}{s^{1-\frac{\beta}{2}-\eps_1}}.
 \end{eqnarray*}
 \item Nonlinear 
 term:
 $$\|\frac{\mathcal{B}_-(y,s)}{1+|y|^3}\|_{L^\infty}\leq \frac{C}{s^{\frac{5}{2}-\eps_1}}(s-\sigma), \quad  \|\mathcal{B}_e(s)\|_{L^\infty}\leq \frac{C}{s^{1-\eps_1}}(s-\sigma)e^{s-\sigma}.$$
 $$\|(1+|y|^{\beta})\mathcal{B}_e(s)\|_{L^\infty}\leq C\frac{1+e^{\frac{\beta}{2}(s-\sigma)}(s-\sigma)}{s^{1-\frac{\beta}{2}-\eps_1}}.$$
 \item 
 Remainder
 term:
 $$\|\frac{\mathcal{C}_-(y,s)}{1+|y|^3}\|_{L^\infty}\leq \frac{C}{s^{\frac{5}{2}-\eps_1}}(s-\sigma), \quad  \|\mathcal{C}_e(s)\|_{L^\infty}\leq \frac{C}{s^{1-\eps_1}}(s-\sigma)e^{s-\sigma}.$$
 $$\|(1+|y|^{\beta})\mathcal{C}_e(s)\|_{L^\infty}\leq C\frac{1+e^{\frac{\beta}{2}(s-\sigma)}(s-\sigma)}{s^{1-\frac{\beta}{2}-\eps_1}}.$$
 \item New term:
 $$\|\frac{\mathcal{D}_-(y,s)}{1+|y|^3}\|_{L^\infty}\leq C(s-\sigma)e^{s-\sigma}e^{-\frac{\gamma}{2}s}, \quad  \|\mathcal{D}_e(s)\|_{L^\infty}\leq C(s-\sigma)e^{s-\sigma}e^{-\frac{\gamma}{2}s} .$$
 $$\|(1+|y|^{\beta})\mathcal{D}_e(s)\|_{L^\infty}\leq C(s-\sigma)e^{-\frac{\gamma}{2}s}(1+s^{\frac{\beta}{2}}e^{\frac{\beta}{2}(s-\sigma)}).$$
\end{enumerate}
\end{lem}
\begin{proof}
Although we have a different shrinking set here, in comparison with our earlier work \cite{AZ1}, the proof follows exactly in the same way. See Lemma 3.16 page 30 in that paper for details.
\end{proof}
In the following statement, we add the bounds of Lemma \ref{lem15} in order to obtain bounds on $v(y,s)$, thanks to the expression \eqref{decomp1}. We also project equation \eqref{eq_v} in order to write differential inequalities satisfies by $v_m$ for $m=0$, $1$ and $2$:
\begin{pro}\label{pro_red} For all $0<\eps_1\le \frac{1}{2}$, there exists $K_7(\eps_1)$ such that for any $K_0\ge K_7(\eps_1)$, there exists $A_7(K_0, \eps_1)\geq 1$ such that for all $A\geq A_7(K_0, \eps_1)$, there exists $s_{07}(K_0, A, \eps_1)$ large enough such that the following holds for all $s_0\geq s_{07}(K_0, A, \eps_1)$:\\
Assume that for some $s_1\geq \sigma\geq s_0$, we have
$$ v(s)\in \mathcal{V}_{A, \eps_1}(s), \quad \mbox{for all} \; s\in [\sigma, s_1], $$
and that $\nabla v$ satisfies the estimates stated in Proposition \ref{prop-reg}.
Then, the following holds for all $s\in [\sigma, s_1]$,
\begin{enumerate}
\item[i)] (ODE satisfied by the positive modes) :
For $m\in \{0,1\}$, we have
$$\displaystyle|v'_m(s)-(1-\frac{m}{2}) v_m(s)|\leq \frac{C}{s^2}.$$
\item[ii)] (ODE satisfied by the null mode) :
We have
$$|v'_2(s)+\frac{2}{s}v_2(s)|\leq\frac{C}{s^3}.$$
\item[iii)] (Control of the negative and outer modes): We have
\begin{eqnarray*}
\displaystyle \|\frac{v_{-}(s)}{1+|y|^3}\|_{L^\infty}\!\!\!&\leq&\!\!\! Ce^{-\frac{s-\sigma}{2}}\|\frac{v_{-}(\sigma)}{1+|y|^3}\|_{L^\infty}\!+\!C\frac{e^{-(s-\sigma)^2}}{s^{\frac{3}{2}}}\|v_e(\sigma)\|_{L^\infty}\!+ \!C\frac{1+s-\sigma}{s^{\frac{5}{2}-\eps_1}},\\
 \|v_{e}(s)\|_{L^\infty}\!\!\!\!\!\!&\leq& \!\!\!Ce^{-\frac{s-\sigma}{p}}\|v_e(\sigma)\|_{L^\infty}\!+ \!Ce^{s-\sigma}s^{\frac{3}{2}}\|\frac{v_{-}(\sigma)}{1+|y|^3}\|_{L^\infty}\!+ \!C\frac{1+(s-\sigma)e^{s-\sigma}}{s^{1-\eps_1}}.
\end{eqnarray*}
\item[iv)]
(Weighted estimate of the outer mode):
\begin{eqnarray*}
 \displaystyle \|(1+|y|^\beta)v_{e}(s)\|_{L^\infty}\!\!\!&\leq& \!\!\!C \displaystyle e^{-\frac{s-\sigma}{2}(\frac{1}{p-1}-\frac{\beta}{2})} \|(1+|y|^\beta)v_{e}(\sigma)\|_{L^\infty}\\
 &+& Ce^{\frac{\beta}{2}(s-\sigma)}s^{\frac{3}{2}+\frac{\beta}{2}}\|\frac{v_{-}(\sigma)}{1+|y|^3}\|_{L^\infty}+ C\frac{1+(s-\sigma)e^{\frac{\beta}{2}(s-\sigma)}}{s^{1-\frac{\beta}{2}-\eps_1}}.
\end{eqnarray*}
\end{enumerate}

\end{pro}
\begin{proof}
The proof of items iii) and iv) is straightforward from Lemma \ref{lem15} and the decomposition \eqref{decomp1}.\\
As for items i) and ii), they follow exactly in the same way as in any paper about the standard heat equation ($\mu=0$), since the effect of the ``new term" is exponentially small, thanks to Lemma \ref{lem_new_term}. See for example Lemma 3.10 page 1293 of Nguyen and Zaag \cite{NZ0}.
 \end{proof}
 
 Now, with Proposition \ref{pro_red}, which gives estimates on the different components of the decomposition \eqref{decomp1}, 
 the proof of Proposition \ref{prop_red} follows exactly as in our earlier work \cite{AZ1}. For that reason, we omit the proof and refer the reader to page 31 of that paper, where the analogous statement (numbered Proposition 3.5) is proved.

\subsection{Proof of the existence of a solution  in $\mathcal{V}_{A, \eps_1}(s)$}
In this subsection, we solve the two-dimensional problem using index theory and we prove the existence of  a solution $v$ of \eqref{eq_v} in $\mathcal{V}_{A, \eps_1}(s)$. More precisely,
we have the following statement:
\begin{pro}[Existence of a solution in the shrinking set]\label{exis} For all $0<\eps_1\le \frac{1}{2}$, there exists  $K_4(\eps_1)\ge 1$ such that  for all $K_0\ge K_4(\eps_1)$, there exists $A_4(K_0, \eps_1)\geq 1$ such that for all $A\geq A_4(K_0, \eps_1)$ there exists $s_{04}(K_0,A, \eps_1)$ such that for all $s_0\geq s_{04}(K_0, A, \eps_1)$, there exists $(d_0, d_1)$ such that if $v$ is the solution of $(\ref{eq_v})$ with initial data at $s_0$, given in $(\ref{ci})$, then $v(s)\in \mathcal{V}_{A, \eps_1}(s)$, for all $s\geq s_0$.
\end{pro}
\begin{proof}
The derivation of this proposition from the reduction to a finite dimensional problem and the  transverse crossing (both given in Proposition \ref{prop_red}) is classical in the literature concerning construction of solutions with prescribed behavior. For that reason, we omit it and refer the reader to page 13 of \cite{AZ1} where an analogous statement is given (namely, Proposition 3.7). The proof in that paper holds here with natural modifications.
\end{proof}

 
  \section{Proof of the main result}
  This section  is dedicated to the proof of 
  Theorem \ref{th1}. 
  \begin{proof}[Proof of Theorem \ref{th1}]
  Consider some $\epsilon_1\in (0, \frac 12]$ and $\alpha \in (0, \frac 12)$. 
  Using Proposition \ref{exis}, 
  we derive
  the existence of a solution $v$ of equation $(\ref{eq_v})$ defined for all $y\in \R$ and $s\geq s_0$, 
  such that $v(s)\in \mathcal{V}_{A, \eps_1}(s)$ for some $s_0\geq 1$ and $A>0$. Moreover, using \eqref{double},
  we have for all $s\ge s_0$,
  $$\|(1+|y|^{\beta})v(s)\|_{L^\infty}+\|(1+|y|^{\beta})\nabla v(s)\|_{L^\infty}\leq \displaystyle \frac{C(A)}{s^{1-\frac{\beta}{2}-\eps_1}}$$
  where $C(A) = CA^2$. For simplicity, we omit the dependence on $A$ in the following.\\
By definition \eqref{v} of $v$, we see that for all $s\ge s_0$
 \begin{equation*}
\|( 1+|y|^{\beta})\big(w(y,s)-\varphi(y,s)\big)\|_{L^\infty}+\|(1+|y|^{\beta})\nabla_y \big(w(y,s)-\varphi(y,s)\big)\|_{L^\infty}\leq \displaystyle \frac{C}{s^{1-\frac{\beta}{2}-\eps_1}},
 \end{equation*}
where $\varphi$ is the profile introduced in $(\ref{n_profile})$.\\
First, we  estimate this profile. We remark that our  "intermediate" profile satisfies:
    $$|\displaystyle f(\frac{y}{\sqrt{s}})|\sim \big(\frac{s}{b|y|^2}\big)^{\frac{1}{p-1}}\; \mbox{as}\; \frac{y}{\sqrt{s}}\to +\infty,$$
and for all $K_0\ge 1$  and $|y|\geq 2K_0\sqrt{s}$, we have $$\displaystyle |\frac{1}{\sqrt{s}}\nabla f(\frac{y}{\sqrt{s}})|\leq \frac{C}{(1+|y|^\beta)s^{\frac{1}{2}-\frac{\beta}{2}}},$$
on the one hand (remember that $\beta < 2/(p-1)$).\\
 On the other hand, the cut-off  function $\chi_0$ satisfies for all $y\in \R$,
$$\displaystyle|\frac{\kappa N}{2ps}\chi_0(\frac{y}{g_\eps(s)})|\leq \frac{C}{(1+|y|^\beta)s^{1-\frac{\beta}{2}-\eps \beta}}, \quad |\frac{\kappa N}{2ps}\frac{1}{g_\eps(s)}\nabla \chi_0(\frac{y}{g_\eps(s)})|\leq \frac{C}{(1+|y|^\beta)s^{\frac{3}{2}-\frac{\beta}{2}-\eps(\beta-1)}}.$$
With the choice of  $\eps\in (0, \min(1, \eps_1\beta^{-1}))$ (see \eqref{choice}), we obtain
$$\displaystyle|\frac{\kappa N}{2ps}\chi_0(\frac{y}{g_\eps(s)})|\leq \frac{C}{(1+|y|^\beta)s^{1-\frac{\beta}{2}-\eps_1}}, \quad |\frac{\kappa N}{2ps}\frac{1}{g_\eps(s)}\nabla \chi_0(\frac{y}{g_\eps(s)})|\leq \frac{C}{(1+|y|^\beta)s^{1-\frac{\beta}{2}-\eps_1}}.$$
Using the above  estimates, we deduce that
\begin{equation*}
\|(1+|y|^{\beta})(w(y,s)-f(\frac y{\sqrt s}))\|_{L^\infty}+\|(1+|y|^{\beta})\nabla_y (w(y,s)-f(\frac y{\sqrt s}))\|_{L^\infty}\leq \displaystyle \frac{C}{s^{1-\frac{\beta}{2}-\eps_1}}.
\end{equation*}
Therefore, thanks to
the similarity variables transformation \eqref{var_sim}, the solution $u$ of equation $(\ref{eq_u})$ defined for all $x\in \R$ and $t\in [0, T)$ satisfies
\begin{equation*}
|(T-t)^{\frac{1}{p-1}}u(x,t)-f(\frac{x}{\sqrt{(T-t)|\log(T-t)|}})|\leq \displaystyle \frac{C}{(1+(\frac{|x|^2}{T-t})^\frac{\beta}{2})|\log(T-t)|^{1-\frac{\beta}{2}-\eps_1}},
\end{equation*}
and 
\begin{eqnarray*}
|(T-t)^{\frac{1}{2}+\frac{1}{p-1}}\nabla u(x,t)&-&\frac{1}{\sqrt{|\log(T-t)|}} \nabla f(\frac{x}{\sqrt{(T-t)|\log(T-t)|}})|\\
&&\leq \displaystyle \frac{C}{(1+(\frac{|x|^2}{T-t})^\frac{\beta}{2})|\log(T-t)|^{1-\frac{\beta}{2}-\eps_1}},
\end{eqnarray*}
and 
identities \eqref{N_u} and \eqref{N_nabla} follow.
In particular,
$\lim_{t\to T} (T-t)^{\frac{1}{p-1}}u(0, t)=(p-1)^\frac{1}{p-1}$ and $u$ blows up at time $t=T$ at the origin.\\

Since our results in item $1)$ of this Theorem are
  better that our earlier result in \cite{AZ1}, the single point
  blow-up property for $u$ and $\nabla u$ follows exactly as in our
  paper \cite{AZ2}. If for $u$, we clearly have from \eqref{N_u} that
  $u(0,t) \to \infty$ as $t\to T$, the question to find some $\xi(t)
  \to 0$ such that $|\nabla u(\xi(t),t)|\to 0$ was left open after our
  previous work. In this new work, we are able to prove it, as we have
  already stated in the third remark following the statement of
  Theorem \ref{th1}, and as we show in the following:\label{rk}
If we choose $\eps_1$ such that $\eps_1\le \frac{1}{2}-\frac{\beta}{2}$, then we have  $$(T-t)^{\frac{1}{p-1}}\sqrt{(T-t) |\log(T-t)|}\nabla u(\sqrt{(T-t)|\log(T-t)|}, t)\sim \nabla f(K_0) \; \mbox{ as}\; t\to T.$$
In particular,
$$\|\nabla u(t)\|_{L^\infty}\geq \frac{C}{(T-t)^{\frac{1}{p-1}}\sqrt{ (T-t) |\log(T-t)|}}. $$
 Thus  $\nabla u$ blows up at time $t=T$ at the origin.
This concludes the proof of  parts   $1)$ and $2)$ of Theorem \ref{th1}.\\

Finally,  for  the proof  of part $3)$ of   Theorem \ref{th1}, the reader can adapt easily the proof in   \cite{AZ2} (section $4$ pages 2618-2622),   where we replace the cut off function in \cite{AZ2} (step 1 page 2619), by $\phi_r(\xi)=\phi (\displaystyle \frac{\xi}{r|\log(T-t_0)|^\alpha})$, $0<\alpha< \frac{1}{2}$ and the bounded in estimate $(4.1)$ in \cite{AZ2} page 2618,  by $ \displaystyle \frac{c}{|\log(T-t_0)|^{1-\eps_1}}$, $0<\eps_1\le \frac{1}{2}$.
\end{proof}

\end{document}